\documentclass{amsart}

\newtheorem{theorem}{Theorem}[section]
\newtheorem{lemma}[theorem]{Lemma}

\theoremstyle{definition}
\newtheorem{defi}[theorem]{Definition}

\newtheorem{corollary}[theorem]{Corollary}

\theoremstyle{remark}

\newcommand{\Z}{\mathbb{Z}}
\newcommand{\C}{\mathbb{C}}
\newcommand{\p}{\mathfrak{P}}
\newcommand{\be}{\begin{equation}}
\newcommand{\ben}{\end{equation}}
\newcommand{\sni}{\underset{i=1}{\overset{n}{\sum}}}
\numberwithin{equation}{section}

\begin{document}

\title{Stritly positive definite functions on compact abelian groups}

\author{Jan Emonds$^*$}
\address{Institut f\"ur Mathematik, Universit\"at Paderborn, D-33098 Paderborn, Germany }
\email{emonds@math.upb.de}
\thanks{$^*$Supported by the DFH and the International Research Training Group DFG-1133 ''Geometry and Analysis of Symmetries''}

\author{Hartmut F\"{u}hr}
\address{Lehrstuhl A f\"ur Mathematik, RWTH Aachen, D-52056 Aachen, Germany}
\email{fuehr@matha.rwth-aachen.de}


\date{\today}



\begin{abstract}
We study the Fourier characterisation of strictly positive definite functions on 
compact abelian groups. Our main result settles the case $G = F \times 
\mathbb{T}^r$, with $r \in \mathbb{N}$ and $F$ finite. The characterisation 
obtained for these groups does not extend to arbitrary compact abelian groups; 
it fails in particular for all torsion-free groups. 
\end{abstract}

\maketitle

\section{Introduction}
Let $G$ be a compact abelian group. A complex-valued function  $f$ on $G$ is 
called positive definite if for all $x_1,\ldots,x_n$ in $G$ 
and all $c_1,\ldots,c_n$ in $ \C\setminus\{0\}$ we have \be \underset{i,j=1}{\overset{n}{\sum}}c_i\overline{c_j}f(x_j^{-1}x_i)\geq 
0.\ben
If the inequality above becomes strict whenever the $x_i$ are distinct, we call $f$ strictly positive definite. 
By $ \p(G)$ we denote the set of continuous positive definite functions on 
$G$ and by $\p^+(G)$ the subset of strictly positive definite 
functions. 

Bochner's theorem \cite[Theorem 1.4.3]{R} provides a neat characterisation of 
the elements of $\p(G)$ via the Fourier transform: $f \in \p(G)$ iff 
$\widehat{f} \ge 0$. For $\p^+(G)$ however, no simple general characterisation 
is known. For the compact group $\mathbb{T}$ of complex numbers with modulus 
one, the question has been studied in \cite{ER,P,Su}, and solved in \cite{ER,P}. 
Partial results for general, not necessarily abelian, compact groups may be 
found in \cite{AP}.


 It is an easy observation to make that in order to decide whether $f \in \p 
(G)$ is in fact in $\p^+ (G)$, only the support of 
$\widehat{f}$ needs to be known; see Theorem \ref{thm:spd_sets} for a precise statement. Thus, 
strict positive definiteness translates to a property of subsets of the dual 
group $\widehat{G}$, and we accordingly call $K \subset \widehat{G}$ {\bf 
strictly positive definite} if it is the support of $\widehat{f}$, for some $f 
\in \p^+ (G)$. The paper compares this notion to another property, called 
ubiquity: A subset $K \subset \widehat{G}$ is called {\bf ubiquitous} if for all 
$H < \widehat{G}$ of finite index and all $ \gamma \in \widehat{G}$, the 
intersection $\gamma H \cap K$ is nonempty. It is fairly easy to see that $K$ is 
ubiquitous if it is strictly positive definite; see Lemma \ref{lem:spd_ubiq} 
below. The chief result of our paper states that the converse is true for $G = F 
\times \mathbb{T}^r$:

\begin{theorem}
 \label{thm:main} Let $G =  F \times \mathbb{T}^r $, and let $K \subset 
\widehat{G}$. Then $K$ is strictly positive definite iff $K$ is ubiquitous. 
\end{theorem}
 
The case $d=1$, $F$ trivial was settled in \cite{ER,P}. The paper \cite{Su} 
established partial results, apparently unaware of the previous source. While 
strictly speaking the results of \cite{Su} are contained in the earlier paper, 
we have found \cite{Su} to be a useful source of ideas; in particular the notion 
of ubiquity is taken from that paper. 

The paper proceeds as follows: Section \ref{sect:prelim} collects general 
remarks and definitions relating to strict positive definiteness. We observe 
that if $\p^+(G)$ is nonempty, then $G$ is metrisable (Corollary 
\ref{cor:spd_metrisable}). We then prove the implication ``strict positive 
definite $\Rightarrow$ ubiquitous'' (Lemma \ref{lem:spd_ubiq}). A closer look at 
the torsion subgroup of $G$ allows to determine interesting classes of examples:  
The converse of \ref{lem:spd_ubiq} is true for all torsion groups, and fails for 
all torsion-free groups. In the final section we focus on the proof of Theorem 
\ref{thm:main}. 

\section{Preliminaries and generalities}
\label{sect:prelim}
Throughout this paper, $G$ denotes a compact abelian group, and $\widehat{G}$ 
its character group. Throughout this section, we will write the group operations 
in $G$ and $\widehat{G}$ multiplicatively. 
In the context of compact groups, Bochner's theorem yields that every function 
$f$ in $ \p(G)$ has a uniformly converging Fourier series with positive 
coefficients. I.e., there is a subset $ K$ of $\widehat{G}$ and a sequence 
$(a_\gamma)_{\gamma\in K}$ of strictly positive numbers such that \be  
\label{eqn:pd_fourseries} f(x)=\underset{\gamma\in K}\sum a_\gamma\gamma(x). 
\ben Let $F=\{x_1,\ldots,x_n\}$ be a subset of $G$ and $c=(c_1,\ldots,c_n)^T$ a 
vector in $\C^n$. A function $p_{c,F}$ on $\widehat{G}$ is defined via \be 
\gamma\mapsto \underset{i=1}{\overset{n}{\sum}}c_i\gamma(x_i)\ben and we call 
such a function a trigonometric polynomial on $\widehat{G}$. Note that the space 
of trigonometric polynomials is closed under addition, multiplication and 
complex conjugation. Furthermore, since characters over abelian groups are 
linearly independent, any trigonometric polynomial arising from pairwise 
different $x_i$ with nonzero coefficients $c_i$ will be nonzero. If $f$ is given 
by (\ref{eqn:pd_fourseries}), then one has \be \label{eqn:spd_sets} 
\underset{i,j=1}{\overset{n}{\sum}}c_i\overline{c_j}f(x_j^{-1}x_i)=\underset{\gamma\in 
K}\sum a_\gamma\left|\sni c_i\gamma(x_i)\right|^2=\underset{\gamma\in K}\sum 
a_\gamma|p_{c,F}(\gamma)|^2.\ben 
In particular, (\ref{eqn:spd_sets}) vanishes iff the trigonometric polynomial 
$p_{c,F}$ vanishes on $K$. This observation motivates the following definition:
\begin{defi}
A subset $K$ of $\widehat{G}$ is called strictly positive definite if there is 
no trigonometric polynomial vanishing on $K$ except for the zero polynomial.
\end{defi}
The above calculations have established the following result:  
\begin{theorem} \label{thm:spd_sets}
A function $f \in \p(G)$ is in $ \p^+(G)$ iff the support of $ \hat{f}$ is 
strictly positive definite.
\end{theorem}

Let us now collect some basic properties of strictly positive definite sets. The 
following arguments will rely mainly on duality theory. In particular, we recall 
the notion of annihilator subgroups: For $M \subset \widehat{G}$, let $M^\bot = 
\{ x \in G: \gamma(x) = 1, \forall \gamma \in M \}$. Likewise, $N^\bot = \{ 
\gamma \in \widehat{G} : \gamma (x) = 1, \forall x \in N \}$ for $N \subset G$. 

\begin{lemma}
 Let $K \subset \widehat{G}$ be strictly positive definite. Then $K$ generates 
$\widehat{G}$.
\end{lemma}
\begin{proof}
 Assume that $H = \langle K \rangle $ is a proper subgroup. Pick a nontrivial 
character $\tilde{\chi}$ of the quotient group $\widehat{G}/ H$, then 
$\chi(\gamma) = \tilde{\chi}(\gamma H)$ defines a character of $\widehat{G}$. By 
Pontryagin duality there exists $x \in G$ such that $\chi(\gamma) = \gamma(x)$. 
The nonzero trigonometric polynomial $p(\gamma) = \gamma(x)-1$ vanishes on $H 
\supset K$, proving that $K$ is not strictly positive definite. 
\end{proof}

\begin{corollary} \label{cor:spd_metrisable}
 $\p^+(G)$ is nonempty iff $G$ is metrisable.
\end{corollary}
\begin{proof}
Note that by \cite[Theorem 2.2.6]{R}, $G$ is metrisable iff $\widehat{G}$ is 
countable. Now if $\p^+(G)$ is nonempty, there exists a strictly positive 
definite $K \subset \widehat{G}$. Since $K$ is the support of a converging 
Fourier series, $K$ is countable. But then $\widehat{G} = \langle K \rangle$ is 
countable.

For the converse, we pick a summable nowhere vanishing family 
$(a_{\gamma})_{\gamma \in \widehat{G}}$ of positive numbers, which exists by 
countability of $\widehat{G}$. Define $f$ according to (\ref{eqn:pd_fourseries}), 
with $K = \widehat{G}$. Now (\ref{eqn:spd_sets}), with $K=\widehat{G}$, implies 
that $f$ is strictly positive definite. 
\end{proof}

Let us next establish the general implication between strict positive 
definiteness and ubiquity. 
The central question of this paper is when the converse of this result holds.
\begin{lemma} \label{lem:spd_ubiq}
 If $K \subset \widehat{G}$ is strictly positive definite, it is ubiquitous.
\end{lemma}
\begin{proof}
First we prove that for proper subgroups $H< \widehat{G}$ of finite index and $ 
\gamma \in \widehat{G}$ there exists a trigonometric polynomial vanishing 
precisely on $\gamma H$: By duality, $H^\bot \cong \left( \widehat{G}/H\right)^{\wedge}$ is finite, and 
thus $H = H^{\bot \bot} = \{ x_1,\ldots,x_n \}^\bot$, with $x_1,\ldots,x_n \in 
G$. Hence, if we define a trigonometric polynomial $p$ by \be 
p(\mu)=\sni|\overline{\gamma(x_i)}\mu(x_i)-1|^2 \mbox{ , } (\mu\in G),\ben we 
find that $p(\mu)=0$ iff $\gamma^{-1} \mu \in \{x_1,\ldots,x_n \}^\bot$, iff 
$\mu\in \gamma H$.

Now, if $K$ is not ubiquitous, then $K \cap \gamma H = \emptyset$ for suitable 
$H$ of finite index, and $\gamma \in \widehat{G}$. Write $\widehat{G} \setminus 
\gamma H = \bigcup_{i=1}^m \mu_i H$, and pick trigonometric polynomials $p_i$ 
vanishing precisely on $\mu_i H$. Then $\prod_{i=1}^m p_i$ is a trigonometric 
polynomial vanishing precisely outside of $\gamma H$. It is therefore nonzero 
and vanishes on $K$, which then cannot be strictly positive definite. 
\end{proof}

Next we characterise finite strictly positive definite subsets. As a byproduct, 
we clarify the case of finite groups.  
\begin{lemma} \label{lem:finite_spd}
 Let $K \subset \widehat{G}$ be finite. Then $K$ is strictly positive definite 
iff $G$ is finite and $K = \widehat{G}$. 
\end{lemma}
\begin{proof}
 Note that by definition, $K$ is strictly positive definite iff the restriction 
map $p \mapsto p|_K$, defined on the space of trigonometric polynomials, has 
trivial kernel. If $G$ is infinite, the space of trigonometric polynomials on 
$\widehat{G}$ is infinite-dimensional, precluding the existence of finite 
strictly positive definite sets. 

Thus, if $K$ is finite, $G$ has to be finite as well, and ubiquity of $K \subset 
\widehat{G}$ implies $K = \widehat{G}$.  The converse is obvious. 
\end{proof}

 Some clarification concerning the converse of Lemma \ref{lem:spd_ubiq} is 
obtained by considering the torsion subgroup $G_t$ of $G$, defined as
\[
 G_t = \{ x \in G : x^n = e_G, \mbox{ for suitable } n \in \mathbb{N} \}~.
\] The torsion subgroup is usually not closed; for instance, the torsion 
subgroup of the torus group is dense. The following observation indicates how 
the torsion subgroup is related to ubiquity.  
\begin{lemma} \label{lem:dual_torsionsubg}
 Let 
\[ H_0 = \bigcap_{H< \widehat{G}, [\widehat{G}:H] < \infty} H~.\]
Then 
\[
  H_0 = G_t^\bot~,~ H_0^{\bot} = \overline{G_t}
\]
\end{lemma}
\begin{proof}
 We first prove $G_t \subset H_0^\bot$: By the isomorphism theorem for groups, 
$\gamma \in H_0$ iff $ \phi(\gamma) = e$, for all group homomorphisms $\phi$ 
with finite image. In particular, if $x \in G_t$, the mapping $\widehat{G} \ni 
\gamma \mapsto \gamma(x) \in \mathbb{T}$ has finite image,
since $\gamma(x)^n = \gamma(x^n) = 1$. It follows for $\gamma \in H_0$ that 
$\gamma(x) = 1$, which means $x \in H_0^\bot$. 

For the proof of $G_t^\bot \subset H_0$ let $H< \widehat{G}$ be of finite index. 
By duality theory, $H^\bot \cong \left( \widehat{G}/H \right)^{\wedge}$ is finite, hence a subgroup of 
$G_t$. But then $G_t^\bot \subset H^{\bot \bot} = H$. Since $H< G$ was chosen 
arbitrary of finite index, it follows that $G_t^\bot \subset H_0$. 

Both inclusions shown so far imply $H_0 = H_0^{\bot \bot} \subset G_t^\bot 
\subset H_0$, and thus $H_0 = G_t^\bot$. The second equality follows from this. 
\end{proof}

%

We now settle the extreme cases $G_t=G$ and $G_t= \{ e \}$. First the good news. 

\begin{theorem}
 Let $G$ be a torsion group, and $K \subset \widehat{G}$. Then $K$ is strictly 
positive definite iff $K$ is ubiquitous.  
\end{theorem}
\begin{proof}
Only the ``if''-part needs to be shown. 
Suppose that $K \subset \widehat{G}$ is not strictly positive definite. Hence 
there is a nontrivial trigonometric polynomial 
\[
 p : \widehat{G} \ni \gamma \mapsto \sum_{i=1}^n c_i \gamma(x_i)
\] of $\widehat{G}$ vanishing on $K$. Since $G$ is a torsion group, $\langle 
x_1,\ldots,x_n \rangle$ is finite, and by duality theory, $H = \{ x_1,\ldots,x_n 
\}^\bot \subset \widehat{G}$ has finite index. Furthermore, for any $ \gamma \in 
\widehat{G}$ and $\eta \in H$, one has 
\[
 p(\gamma \eta) = \sum_{i=1}^n c_i \gamma(x_i) \eta(x_i) = p(\gamma) ~,
\] implying that $p$ vanishes on an $H$-invariant set. In particular, since $p$ 
is nonzero and vanishes on $K$, $\widehat{G} \setminus K$ contains an $H$-coset. 
Thus $K$ is not ubiquitous. 
\end{proof}

The theorem applies to groups of the form $G = \prod_{i=1}^\infty F_i$, with 
finite groups $F_i$ of bounded order. The other extreme provides a whole class 
of examples for which the converse of Lemma \ref{lem:spd_ubiq} fails. 
\begin{theorem} \label{thm:torsionfree}
 Suppose that $G$ is nontrivial and torsion-free. Then every nonempty subset of 
$\widehat{G}$ is ubiquitous, but finite subsets are not strictly positive 
definite. 
\end{theorem}
\begin{proof}
 Suppose $G_t$ is trivial. With $H_0$ as defined in Lemma \ref{lem:dual_torsionsubg}, 
one obtains from \ref{lem:dual_torsionsubg} that $\widehat{G}/H_0$ is trivial 
also. Hence $\widehat{G}$ has no proper finite index subgroups, and then every 
nonempty subset is ubiquitous. Since $G$ is torsion-free and nontrivial, it is 
infinite, and then Lemma \ref{lem:finite_spd} implies that no finite $K \subset 
\widehat{G}$ is strictly positive definite. 
\end{proof}

This result applies, for instance, to the group $\mathbb{Z}_p$ of $p$-adic 
integers. 


\section{The case $G = F \times \mathbb{T}^r$}

The remainder of the paper is reserved for $G = F \times \mathbb{T}^r$. We 
identify the dual group of $G$ in the canonical way with $F \times 
\mathbb{Z}^r$, and write the latter additively. 
At first we will deal with $\mathbb T^r$ separately. Here we will need a fairly 
deep theorem from number theory. 
\subsection{Products of $ \mathbb{T}$}
 We start with two lemmata that will be needed in the proof of the main 
result. The first one is a fact from elementary group theory.
\begin{lemma} \label{lem:index_fi}
Every finite intersection of subgroups of finite index is a subgroup of finite 
index as well.
\end{lemma}
\begin{proof}
This follows by induction and 
\begin{enumerate}
\item If $A$ and $B$ are subgroups of a group $G$, then $[B:A\cap B] \le [G:A]$. 
\item If $B<G$ and $A<B$, then $[G:A] = [G:B][B:A]$. 
\end{enumerate}
\end{proof}
\begin{lemma} \label{lem:fi_supergrp}
Let H be subgroup of infinite index in $\Z^r$ and $y\in \Z^r\setminus H$. There 
is a subgroup $G$ of finite index such that $H\subset G$ and $y\notin G$.
\end{lemma}
\begin{proof}
Let $H$ be a subgroup of infinite index in $\Z^r$ and $y\notin H$. Then there is 
basis $x_1,\ldots,x_r$ of $\Z^r$ and some $\alpha_1,\ldots, \alpha_r$ in 
$\Z\setminus\{0\}$ such that $ \alpha_1x_1,\ldots,\alpha_sx_s$ form a basis of 
$H$ and $(\underset{i=1}{\overset{s}{\bigoplus}}\Z x_i)/H\cong 
\underset{i=1}{\overset{s}{\bigoplus}}(\Z/\alpha_i\Z)$, see \cite[2.9.2]{B}. In 
particular, we have $s<r$. Now $y$ can be expressed as $y=\underset{i=1}{\overset{r}{\sum}}\beta_ix_i$ 
with unique entire numbers $\beta_1,\ldots,\beta_r$. If
$\beta_{s+1}=\ldots=\beta_{r}=0$ we define $G$ to be the subgroup generated by $ \alpha_1x_1,\ldots,\alpha_sx_s,x_{s+1},\ldots,x_r$ and note that $G\supset H$ is of finite index, see 3.1, and $y\notin G$ due to the uniqueness of the coefficients. If $\beta_i\neq 0$ for some $i\in\{s+1,\ldots,r\}$ then $G$ is defined as the subgroup generated of $\alpha_1x_1,\ldots \alpha_rx_r$ with $\alpha_{s+1},\ldots,\alpha_r$ in $\Z\setminus\{0\}$ and $\alpha_{i}$ not a divisor of $\beta_{i}$, then once again $H$ is a subset $G$, $G$ is of finite index and $y$ is not an element of $G$ by construction. 
\end{proof}
The main device for showing sufficiency of ubiquity is the following theorem due 
to Laurent, see \cite{L}. For a partition $ \mathcal{P}$ of $\{1,\ldots,n\}$ we 
write $\gamma\in Null(p^\mathcal{P})$ if \be 0=\underset{k\in P}\sum 
c_k\gamma(x_k)\ben for all $P\in \mathcal{P}$. Clearly, it is 
$Null(p^\mathcal{P})\subset Null(p)=p^{-1}(\{0\})$. A partition $\mathcal{P}'$ 
is called finer than $\mathcal{P}$, if $\mathcal{P}'$ is a partition of 
$\{1,\ldots,n\}$ and for all $P'\in \mathcal{P}'$ there is $P\in \mathcal{P}$ 
such that $P'\subset P$. In short, we write $\mathcal{P}'<\mathcal{P}$, if 
$\mathcal{P}'\neq \mathcal{P}$ and $\mathcal{P}'$ is finer than $\mathcal{P}$.  
Furthermore, $\gamma\in Null(p)$ is called maximal with respect to 
$\mathcal{P}$, if $\gamma\in Null(p^\mathcal{P})$ and $\gamma\notin 
Null(p^{\mathcal{P}'})$ for every $\mathcal{P}'<\mathcal{P}$.\\ By 
$H_\mathcal{P}$ we denote the subgroup of $\Z^r$ defined by \begin{eqnarray*} 
H_\mathcal{P} & = & \underset{P\in \mathcal{P}}\bigcap \{\gamma\in\Z^r : 
\gamma(x_k)=\gamma(x_{l}) \mbox{ for } k,l\in P\}\\ & = & \underset{P\in 
\mathcal{P}}\bigcap \{\gamma\in \Z^r: \gamma(x_kx_{l}^{-1})=1 \mbox{ for } 
k,l\in P\} \\ & = & \underset{P\in \mathcal{P}}\bigcap \{x_kx_{l}^{-1}: k,l\in 
P\}^\bot .\end{eqnarray*}
Finally, let $S_\mathcal{P}$ be the set of $\gamma\in Null(p)$, which are 
maximal with respect to $\mathcal{P}$. Then one has the following theorem due to 
Laurent, [L].
\begin{theorem}
$S_\mathcal{P}$ is a finite union of $H_\mathcal{P}$-co-sets.
\end{theorem}
This theorem can be regarded as a generalisation of a number theoretical result 
of Skolem, Mahler and Lech on linear recurrences, see [P]. We are interested in 
$Null(p)$, so we need only a corollary. By definition we have  \be Null(p) 
\supset\underset{\mathcal{P} \mbox{ partition of } \{1,\ldots,n\} }\bigcup 
S_\mathcal{P}  .\ben Conversely, $\gamma\in Null(p)$ implies $\gamma\in 
Null(p^\mathcal{P})$, where $\mathcal{P}=\{\{1,\ldots,n\}\}$. Suppose now 
$\gamma\notin S_{\mathcal{P}}$ for every partition $\mathcal{P}$. That is, for 
every partition $\mathcal{P}$ the fact $\gamma\in Null(p^\mathcal{P})$ implies  
$\gamma\in Null(p^{\mathcal{P}'})$ for some $\mathcal{P}'<\mathcal{P}$. But 
there are only finitely many partition of $\{1,\ldots,n\}$, where 
$\mathcal{P}_0=\{\{1\},\ldots,\{n\}\}$ is the finest partition with respect to 
$<$. In particular, $\gamma\in Null(p^{\mathcal{P}_0})$ but $\gamma \notin 
S_{\mathcal{P}_0}$ leads to a contradiction. Hence, it follows $\gamma \in 
S_\mathcal{P}$ for some partition $\mathcal{P}$. That is,
\be Null(p) =\underset{\mathcal{P} \mbox{ partition of } \{1,\ldots,n\} }\bigcup 
S_\mathcal{P} \ben
 and we get:  
\begin{corollary} \label{cor:laurent}
Let $p$ be a nontrivial trigonometric polynomial on $\Z^r$. There are finitely 
many subgroups $G_1,\ldots,G_n$ of $\Z^r$ and $x_1,\ldots,x_n\in\Z^r$ such that 
\be Null(p)=\underset{i=1}{\overset{n}{\bigcup}} x_i+G_i.\ben
\end{corollary}
\begin{theorem} \label{thm:ubiq_spd_torus}
Let $K$ be a subset of $\Z^r$. If K is ubiquitous then it is also strictly 
positive definite.
\end{theorem}
\begin{proof} Assume that $K$ is not strictly positive definite.  Then there 
exists a non-zero trigonometric polynomial $p$ such that $K \subset S$, where 
$S$ denotes the set of its zeros. By Corollary \ref{cor:laurent} we know that 
$S$ can be written as $\underset{i=1}{\overset{n}{\bigcup}}\gamma_i+H_i$ for 
some $\gamma_1,\ldots,\gamma_n$ in $ \Z^r$ and subgroups $H_1,\ldots,H_n$. 
Without loss we can assume that $H_1,\ldots,H_m$ are of finite and 
$H_{m+1},\ldots,H_n$ are of infinite index. Since $p$ is non-zero there is some 
$\gamma$ in $\Z^r\setminus S$. But then \be H'=\underset{i=1}{\overset{m}{\bigcap}}H_i\ben 
satisfies \be \gamma+H'\cap \underset{i=1}{\overset{m}{\bigcup}}\gamma_i+H_i=\emptyset\ben 
and is of finite index, see \ref{lem:index_fi}. Furthermore, for every 
$i=m+1,\ldots,n$ we pick by virtue of Lemma \ref{lem:fi_supergrp} a subgroup 
$I_i$ of finite index such that $H_i\subset I_i$ and $\gamma - \gamma_i \notin 
I_i$. Now we put \be H=H'\cap \underset{i=m+1}{\overset{n}{\bigcap}}I_i,\ben 
then $H$ is still of finite index.
Furthermore, for each $i \in \{ 1, \ldots, m \}$,
\begin{equation*}
 \gamma + H \cap \gamma_i + H_i \subset \gamma + H_i \cap \gamma_i + H_i = 
\emptyset
\end{equation*} by choice of $\gamma$, whereas for  $i \in \{ m+1, \ldots, n 
\}$,
\begin{equation*}
 \gamma + H \cap \gamma_i + H_i \subset \gamma + I_i \cap \gamma_i + I_i = 
\emptyset
\end{equation*}
by choice of $I_i$. 
Hence finally, 
 \be \gamma + H \cap K \subset \gamma+H\cap S=\emptyset ~, \ben   
which shows that $K$ is not ubiquitous. 
\end{proof}
\subsection{Strict Positive Definiteness over Direct Products}
It remains to combine the results for the factors $F$ and $\mathbb{T}^r$, 
obtained in Lemma \ref{lem:finite_spd} and Theorem \ref{thm:ubiq_spd_torus} 
respectively. The following somewhat technical result illustrates that the 
transfer of results for the factors to the product group is not entirely 
trivial. 
\begin{theorem} \label{thm:spd_dirprod}
Let $G = G_1 \times G_2$, with compact groups  $G_1$ and $G_2$. Let  $K\subset 
\widehat{G_1} \times \widehat{G_2}$. For $\gamma\in \widehat{G_1}$ let \be 
K_2(\gamma)=\{\omega:(\gamma,\omega)\in K\}\ben and \be K_1=\{\gamma:K_2(\gamma)\mbox{ 
strictly positive definite}\}.\ben Finally, let \be \widetilde{K}=\underset{\gamma\in 
K_1}\prod\{\gamma\}\times K_2(\gamma)\ben
If $K_1$ is strictly positive definite, then also $\widetilde K$ and in 
particular $K$.
\end{theorem}
\begin{proof}
For $\gamma\in K_1$ and $\omega\in K_2(\omega)$ let positive real numbers 
$a_\gamma$ resp. $b_\omega$ be given such that $\underset{\gamma\in K_1}\sum 
a_\gamma\left (\underset{\omega\in K_2(\gamma)}\sum b_\omega\right)$ is 
convergent. 
If we put $f=\underset{\gamma\in K_1}\sum a_\gamma\gamma\left 
(\underset{\omega\in K_2(\gamma)}\sum b_\omega\omega\right)=\underset{(\gamma,\omega)\in 
\tilde K}\sum a_\gamma b_\omega\gamma\omega$, then $f$ converges absolutely and 
unconditionally on $G_1\times G_2$ by Fubini's theorem, see \cite[(21.13)]{HS}. 
Now suppose that for distinct $z_1=(x_1,y_1),\ldots,z_n=(x_n,y_n)$ in $G_1 
\times G_2$ and some complex $c_1,\ldots,c_n$ we have
\begin{equation}
0=\underset{i,j=1}{\overset{n}{\sum}}c_i\overline{c_j}f(z^{-1}_jz_i)=\underset{(\gamma,\omega)\in 
\tilde K}\sum a_\gamma b_\omega\left|\underset{i=1}{\overset{n}{\sum}}c_i\gamma(x_i)\omega(y_i)\right|^2. 
\end{equation}
Without loss we can assume that $y_1,\ldots,y_m$ are distinct and we put 
$I_l=\{k:y_k=y_l\}$. Thus $I_1,\ldots,I_m$ form a partition of $\{1,\ldots,n\}$ 
and  
(3.12) reads for $(\gamma,\omega)\in \tilde{K}$:
\be 0=\underset{i=1}{\overset{n}{\sum}}c_i\gamma(x_i)\omega(y_i)=\underset{l=1}{\overset{m}{\sum}}\left(\underset{k\in 
I_l}\sum c_k\gamma(x_k)\right)\omega(y_l)\ben
Since $K_2(\gamma)$ is strictly positive definite this implies $0=\underset{k\in 
I_l}\sum c_k\gamma(x_k)$ for $(\gamma,\omega)\in \tilde{K}$ and $l=1,\ldots,m$. 
But this again leads to $c_k=0$ for $k\in I_l$ and $l=1,\ldots,m$, since $K_1$ 
is strictly positive definite. 
\end{proof}
We do not know of an exhaustive characterisation of strictly positive definite subsets of the product group $\widehat{G}_1 \times \widehat{G}_2$ in terms of strictly positive definite subsets of $\widehat{G}_1$ and $\widehat{G}_2$. It is fairly easy to see that the sufficient condition of the theorem is not necessary: For a counterexample, consider the case $G_1 = G_2 = \mathbb{T}$. Let 
\[ K = \bigcup_{n=1}^\infty \{ n \} \times \{ -n,\ldots,n \} ~,\]
and let $K_1$ be defined as in the theorem. Then Lemma \ref{lem:finite_spd} implies that $K_1 = \emptyset$. But $K$ is strictly positive definite, which can be easily seen by applying the theorem with the roles of $\widehat{G}_1$ and $\widehat{G}_2$ interchanged, and using the observation that 
\[
 K = \bigcup_{m=-\infty}^\infty \{ k : k \ge |m| \} \times \{ m \}~.
\] One could formulate a version of the theorem that covers this example as well, e.g. by introducing a condition that is symmetric with respect to the roles of $\widehat{G}_1$ and $\widehat{G}_2$. More generally,  since strict positive definiteness is clearly preserved by the action of a group automorphism, the condition would have to be invariant under automorphisms of $\widehat{G}_1 \times \widehat{G}_2$ as well. The counterexample illustrates that the sufficient condition is not invariant under the automorphism $(\gamma,\omega) \mapsto (\omega,\gamma)$. It seems open to us whether a clean-cut characterisation working for all product groups is available. 
\subsection{Proof of Theorem \ref{thm:main}}
By Lemma \ref{lem:spd_ubiq} it remains to prove the ``if''-direction. Assume 
that $K \subset F \times \mathbb{Z}^r$ is ubiquitous, and define $K_2(\gamma)$, 
for arbitrary $\gamma \in F$, according to Theorem \ref{thm:spd_dirprod}. It 
suffices to show, for all $\gamma \in F$, that $K_2(\gamma) \subset 
\mathbb{Z}^r$ is strictly positive definite. 

 Suppose $\gamma\in F$, $H$ a subgroup of finite index in $\Z^r$ and $\omega\in 
\Z^r$. Then $\{e\}\times H$ is a subgroup of finite index in $ F\times\Z^r$ and 
by assumption the intersection  \be K\cap (\gamma,\omega)(\{e\}\times 
H)=K\cap\{\gamma\}\times \omega H=\{(\gamma,\chi):\chi\in K_2(\gamma)\cap \omega 
H\}\ben is not empty. Hence, the projection thereof to the second variable is 
also not empty and this is nothing but the intersection of $K_2(\gamma)$ with 
$\omega H$. That is, $K_2(\gamma)$ is ubiquitous, for arbitrary $\gamma\in F$. 
Now Theorem \ref{thm:ubiq_spd_torus} implies that $K_2(\gamma)$ is strictly 
positive definite, and we are done. \hfill $\Box$

\bibliographystyle{amsplain}

\end{document}